\documentclass[11pt]{amsart}
\usepackage{epsfig,latexsym,amsfonts,amssymb,amsmath,amscd}
\usepackage{amsthm}
\usepackage[mathscr]{eucal}
\usepackage{mathrsfs}
\usepackage{oldgerm,units,color}
\makeatletter
\def\section{\@startsection{section}{1}%
  \z@{1.1\linespacing\@plus\linespacing}{.8\linespacing}%
  {\normalfont\Large\scshape\centering}}
\makeatother

\theoremstyle{plain}

\newtheorem*{thmA}{Theorem A}
\newtheorem*{thmB}{Theorem B}

\newtheorem*{conj*}{Root Groups Conjecture}
\newcommand{\etype}[1]{\renewcommand{\labelenumi}{(#1{enumi})}}
\newtheorem*{thm1.2}{(1.2) Theorem}
\newtheorem*{thm1.3}{(1.3) Theorem}
\newtheorem*{thm1.4}{(1.4) Theorem}
\newtheorem*{prop*}{Proposition}

\def\wk{{\operatorname{wk}}}

\def\Miy{\operatorname{Cl}}

\def\comm{{\operatorname {comm} }}

\newtheorem{prop}{Proposition}[section]

\newtheorem{thm}[prop]{Theorem}

\newtheorem{cor}[prop]{Corollary}
\newtheorem{lemma}[prop]{Lemma}

\theoremstyle{definition}
\newtheorem{Def}[prop]{Definition}
\newtheorem*{Def*}{Definition}

\newtheorem{Defsnot}[prop]{Definitions and notation}
\newtheorem{example}[prop]{Example}

\newtheorem{note}[prop]{Note}

\newtheorem*{notation*}{Notation}
\newtheorem{remark}[prop]{Remark}

\newtheorem*{remark*}{Remark}

\newcommand{\calb}{\mathcal{B}}

\newcommand{\calg}{\mathcal{G}}

\newcommand{\ff}{\mathbb{F}}

\newcommand{\uu}{\mathbb{U}}

\newcommand{\Cl}{\operatorname{Cl}}

\newcommand{\ga}{\alpha}
\newcommand{\gb}{\beta}
\newcommand{\gc}{\gamma}

\newcommand{\gd}{\delta}

\newcommand{\gl}{\lambda}

\newcommand{\gvp}{\varphi}
\newcommand{\gr}{\rho}

\newcommand{\gt}{\tau}

\newcommand{\charc}{{\rm char}}

\newcommand{\sminus}{\smallsetminus}
\newcommand{\lan}{\langle}
\newcommand{\ran}{\rangle}

\newcommand{\Aut}{{\rm Aut}}

\def\eroman{\etype{\roman}}

\newcommand{\half}{\frac{1}{2}}

\numberwithin{equation}{section}

\begin{document}
\title[Frobenius forms on weak  PAJs]{Frobenius forms on weakly primitive axial algebras of Jordan type}
\author[Louis Halle Rowen , Yoav Segev]
{Louis Rowen$^*$\qquad Yoav Segev}

\address{Louis Rowen\\
         Department of Mathematics\\
         Bar-Ilan University\\
         Ramat Gan\\
         Israel}
\email{rowen@math.biu.ac.il}
\address{Yoav Segev \\
         Department of Mathematics \\
         Ben-Gurion University \\
         Beer-Sheva 84105 \\
         Israel}
\email{yoavs@math.bgu.ac.il}
\thanks{$^*$The first author was supported by the ISF grant 1994/20 and the Anshel Pfeffer Chair}



\begin{abstract}
In  a previous paper we studied   ``weakly primitive axial algebras'' with respect to more general fusion rules,
for which  at least one axis satisfies the fusion rules. In this continuation, a concise description is  provided of the $2$-generated algebras we obtained in that paper and we show that in certain cases, the algebras have a Frobenius form, a key tool in their study.
%continue to examine their structure, add some noncommutative examples, and
% obtain   a Frobenius form in various instances, thereby yielding $A_0(a)^2 \subseteq A_0(a)$.
\end{abstract}
\date{\today}
 \maketitle

%\tableofcontents
%%%%%%%%%%%%%%%%%%%%%%%%%%%%%%%%%%%%%%%%%%%%%%%%%%%
%sec1
\section{Introduction}
%%%%%%%%%%%%%%%%%%%%%%%%%%%%%%%%%%%%%%%%%%%%%%%%%%%%
 This paper continues our study of  ``axial algebras,''
generated by ``axes,'' i.e., semisimple idempotents. In \cite{RoSe5} we considered noncommutative 2-generated algebras satisfying certain fusion rules to be reviewed below. The   main result of this paper is to use those results  to define a  Frobenius form. This is significant because the structure of algebras becomes much more tractable when they are endowed with Frobenius forms, as exemplified in \cite[Appendix~A]{ChG}, and as we shall see in proving that $A_0(a)^2 \subseteq A_0(a)$ and $A_{\mu,\nu}(a)$ is in the radical of the form for $\mu\ne \nu$.

%%%%%%%%%%%%%%%%%%%%%%%%%%%%%%%%%%%%%%%%%%
%1.1
\begin{Defsnot}\label{not1}$ $

%%%%%%%%%%%%%%%%%%%%%%%%%%%%%%%%%%%%%%%%%%%%%%%%%
We review some definitions from \cite{M} and \cite{RoSe5} and add a few more.
\begin{enumerate}
%1 More precisely,
\item $A$ always denotes an  algebra {\it (possibly not commutative)},
over a field~$\ff$ of characteristic $\ne 2.$ In this paper  $\gl,\gd \in \ff \setminus \{ 0, 1  \}$ always.

\item
For  $y\in A,$ write $L_y$ for the left  multiplication map
$z\mapsto yz$, and  $R_y$ for the right  multiplication map
$z\mapsto zy$. $A_\mu(L_y)$ denotes the $\mu$-eigenspace of $L_y$,
for $\mu\in \ff$, which is permitted to be $0$; similarly for $A_\nu(R_y).$  We let $A_{\mu,\nu}(y)=A_{\mu}(L_y)\cap A_{\nu}(R_y).$

 \item
 A {\it left axis}
$a$ is an %nonzero
idempotent  for which $A$ is a direct sum of its
left eigenspaces with respect to $L_a$. A  {\it right axis} is
defined analogously. A {\it  (2-sided)  axis} $a$  is a left axis
which is also a    right axis, for which $L_a R_a = R_a L_a$. When $a$ is an axis, if $S_L: = S_L(a)$ is the set of left eigenvalues and $S_R:= S_R(a)$ is  the set of right eigenvalues, then $A$ decomposes as a direct sum of subspaces
\begin{equation}
    \label{dec18} A = \oplus _{\mu\in S_L, \ \nu\in S_R} A_{\mu,\nu}(a). \end{equation}
%4
 The pairs of eigenvalues $(1,1)$ (and possibly  $(0,0)$) play a particularly significant role.
We write
\[
S(a):=S_L(a)\times S_R(a)\setminus\{(0,0), (1,1)\}.
\]
  $A_{\mu,\nu }(a)$, for $(\mu,\nu)\in S$, is called a {\it two-sided eigenspace}. When $a$ is understood,
 We write  $S$  for  $S(a)$,  and {\it we call $a$ an $S$-axis}.
We also write
$A_\mu(a)$ for $A_{\mu,\mu}(a).$ %Define $A_1(a)+A_0(a)  = A_1(a)+A_0(a) .$
For a subset $S'\subseteq S,$ and an element $y\in A,$ we write $y_{\mu,\nu}$ for the projection of $y$ on $A_{\mu,\nu}(a),$ and we let
\[
\textstyle{A_{S'}(a):=\sum_{(\mu,\nu)\in S'}A_{\mu,\nu}(a),\text{ and }y_{S'}=\sum_{(\mu,\nu)\in S'}y_{\mu,\nu}.}
\]
 Thus, $A$~decomposes as a direct sum of eigenspaces
\begin{equation}
    \label{dec17} A= A_1(a) \oplus A_0(a) \oplus A_S(a),
\end{equation}
%where  $A_S(a)=  \oplus  \{A_{\mu,\nu}(a): (\mu,\nu)\in S\}.$
(one could have $S(a)=\emptyset$), and for $y\in A$ we write $$y = y_1 +y_0 +  y_S,$$ where $y_i\in A_i(a)$, $i=0,1,$ and $y_S\in A_S(a),$ so that $y_S=\sum_{(\mu,\nu)\in S}y_{\mu,\nu}.$

 Let
 \begin{align*}
     %  &S_y^\circ(a) = \{ (\mu,\nu)\in S(a): y_{\mu,\nu}(a)\ne 0\}.\\
       &S^\circ(a) = \{ (\mu,\nu)\in S: A_{\mu,\nu}(a)\ne 0\}.\\
     &S^{\dagger}(a)=\{(\mu,\nu)\in S^{\circ}(a): \mu\ne\nu\}.\\
   & S^{\comm}(a)=  \{  (\mu,\mu):(\mu,\mu)\in S^\circ(a)\}.
   %  &\bar S(a) = \{(\mu,\nu) \in S^\dag(a) :b_{\mu,\nu}^2\ne 0\}.
\end{align*}
Thus $S^\circ(a)= S^{\dagger}(a) \cup   S^{\comm}(a).$

Also define the  {\it complementary set} $\tilde S:= \{(\nu,\mu) : (\mu,\nu) \in S\}.$

\item An $S$-axis $a$ is  {\it
 commutative} if $S^\dagger = \emptyset.$ (Note that this implies that $a$ commutes with all elements in $A.$)
 Otherwise we say that   $a$ is {\it
noncommutative}.

\item  The  {\it fusion rules} in this paper are comprised of the rules:
\begin{enumerate}
    \item $A_1(a)+A_0(a)$ is a subalgebra of $A$;
\item
 $A_{\mu,\nu} (a)$ is
  a left and right $(A_1(a)+A_0(a))$-module.
\item    $A_{\mu,\nu}(a)A_{\mu',\nu'}(a) \subseteq \delta_{\mu,\mu'} \delta_{\nu,\nu'}( A_1(a)+A_0(a)),  $ $\forall (\mu,\nu),(\mu',\nu')\in~S,$ where $\delta$ denotes the Kronecker delta.
\end{enumerate}

In particular, for each nonempty subset $S' \subseteq S,$ the fusion rules imply a
   grading of $A$ as a $\mathbb Z_2$-graded algebra, i.e.,
\begin{equation}\label{dec119}
    A=\overbrace{A_{1}(a) \oplus A_{0}(a)\oplus  A_{S\setminus S'}(a)}^{\text {$(+)$-part}}) \oplus \overbrace{ A_{ S'}(a)}^{\text {$(-)$-part}}. \footnote{Fusion rules (a) and (b) are widespread, cf.~\cite{MS}, but (c) generalizes the definition of Jordan type in \cite{HRS}.}
\end{equation}

 \item An  $S$-axis  $a$ satisfying the fusion rules is of {\it
Jordan type} if
\begin{equation}
    \label{fr} A_0(a)^2 \subseteq A_0(a).
\end{equation}

\item An axis $a$ is {\it weakly primitive} if $A_1(a) = \ff a.$
 An axis $a$ is  {\it left primitive}
if $A_1(L_a) = \ff a.$ Similarly for  {\it right primitive}. An axis
which is both left and right primitive is called {\it
primitive}.

\item An axis  $a$ is {\it special}  if $a$ is weakly primitive  satisfying the fusion rules, such that for any idempotent $b,$ the subalgebra $A=\lan\lan a,b\ran\ran$ satisfies $\dim A_0(a)\le 1.$

\medskip

 \item An axis  $b$ is of {\it
Jordan type} $\eta$ if  $b$ is  an    $S$-axis of {
Jordan type} with one of the following:
\begin{enumerate}
   \item   $S^\comm(b)=\{(\eta,\eta)\}$ for some  $\eta \ne 0,1.$
   \item $A = \langle\langle a,b\rangle\rangle$,
   $S^{\comm}=\emptyset,$ $\eta\in\{0,1\},\ \dim A_{\eta}(b)=2,$ and $\dim A_{1-\eta}(b)=1.$
\end{enumerate}
 % $\eta = 0$ and $S^\comm(a) = \emptyset$  and $\dim A_0(a)=2.$

    % $a$
    %is a weakly primitive axis,
% \item For $\eta = 1,$   $\dim A_1(a) =2$, and $ S^\circ(a) = S^\dagger(a)$.
\item If $S^\comm(b) = \emptyset$ and
\begin{enumerate}
    \item
if $b$ is weakly primitive then we say that  $b$ is of {\it
Jordan type}~$\emptyset.$

\item if $A = \langle\langle a,b\rangle\rangle$, $\dim A_0(b)=2=\dim A_1(b),$ then we say that  $b$ is of {\it
Jordan type}~$\{0,1\}.$
\end{enumerate}

\medskip

%%Changed here

\item For a given axis $a\in A,$   any element $y\in A$ has a unique decomposition  $y =    y_1 +y_0 +  y_{ S^\comm(a)} + y_{S^\dagger(a)}, $ where  $y_1\in A_1(a)$ and  $y_0\in A_0(a); $  %$y_{\mu,\mu} \in A_{\mu,\mu}(a), $
when $a$ is weakly primitive we  write $y_1 = \gvp_a(y) a$ where $\gvp_a(y) \in \ff.$
\end{enumerate}
\end{Defsnot}

    \begin{lemma}
        If $S' \subseteq S$ then $A_1 (a)\oplus A_0 (a) \oplus A_{S'}(a)$ is a subalgebra of $A.$
    \end{lemma}
\begin{proof}
    By the fusion rules.
\end{proof}

\begin{Def}$ $
\begin{enumerate}
\item
 $X\subseteq A$ denotes a set of  axes. $\langle\langle X \rangle\rangle$ denotes the
subalgebra of $A$ generated by $X$.
%We say that $A$ is an {\it axial algebra} over
If $A= \langle\langle X \rangle\rangle,$ we say that $X$ {\it generates} $A$;
 if   $|X|=n$  we say that $A$ is
{\it $n$-generated}.

    \item
    An axial algebra $A$ is a {\it
    (weak) PAJ}, if $A= \lan\lan X \ran\ran,$  where $X$ is a set of
(weakly) primitive axes of  Jordan type.
\begin{enumerate}
 \item
 $ A$ is of {\it Jordan type $\eta$} if each $a\in X$ has Jordan type $\eta$.

   \item
A commutative {PAJ}  is called a {\it CPAJ}.
\end{enumerate}

\item A set $X$ of axes is {\it homogeneous} if $S^{\circ}(b)=S^{\circ}(c)$ or $S^{\circ}(b)=\tilde S^{\circ}(c),$ for $b,c\in X.$
%\item An axial algebra  is {\it homogeneous} if it is generated by a homogeneous set $X$.
\end{enumerate}
    \end{Def}

\subsection{Review of the Classification of 2-generated weak PAJ's, \\ from~\cite{RoSe5}}$ $

\cite[Theorem A]{RoSe5} characterized all  2-generated-CPAJ's $\lan\lan a,b\ran\ran$.

\begin{example}\label{class}
Next, we list the possibilities for  $A= \lan \lan a,b \ran\ran$,  as classified in~\cite{RoSe5},  where $a$ is a special axis,   and $b$ is an axis satisfying the fusion rules.  %Throughout our previous investigations we have assumed that $a$ is an axis with $\dim A_0(a)\le 1.$
(It follow by \cite[Lemma~3.7]{RoSe5} that $\dim A_{\mu,\nu}(a)\le 1$ for all $(\mu,\nu)\in S(a).$)

  In   (iii)--(vi) below, both $a$ and $b$ are of Jordan type $\emptyset,$ $\{\half\},$ or $\{2\}$. In the last three examples, $b$ is not weakly primitive. We write $S(a) $  and $S(b)$ to avoid ambiguity.
\begin{enumerate}\eroman
    \item  The list of  \cite[Theorem~1.1]{HRS}.
  We call these axial algebras {\it HRS algebras}, which are commutative.  In each case, $S(a)$ and $S(b)$  have order $\le 1.$ Both axes are primitive of Jordan type. If $S(a)=\{\eta\}, \eta\notin\{0,1\},$    then necessarily   $S(b)= \{\eta\}$  or  $S(b)= \{1-\eta\}$.

\item
A 4-dimensional CPAJ, cf.~\cite[Example~4.5]{RoSe5}.
%Let $S^{\circ}(a)=\{(\half,\half), (2,2)\}.$
Let
\[
A=\ff a+\ff b_0+\ff b_{2,2}+\ff b_{\half,\half}, \qquad b=a+b_0+b_{2,2}+b_{\half,\half},
\]
%with $A$ commutative, and
\[
\textstyle{b_{\half,\half}^2=0,\quad b_{2,2}^2=-\half b_0,\quad b_{\half,\half}b_{2,2}=0.}
\]
\[
\textstyle{b_0^2=\frac{3}{2}b_0,\qquad  b_0b_{\half,\half}=0,\qquad b_0b_{2,2}=-\frac{3}{2}b_{2,2}.}
\]
\[
\textstyle{a=b+a_0+a_{2,2}+a_{\half,\half},\quad a_0\in A_0(b),\quad a_{\eta,\eta}\in A_{\eta}(b).}
\]
 Hence the axis $b$ is primitive and $S(a)=S(b)=\{(\half,\half),(2,2)\}.$
 \smallskip

\item
\cite[Example~4.13]{RoSe5} $S^{\circ}(a)\subseteq\{(\mu,\nu)\in\ff\times\ff: \mu+\nu=1\},$ with $(\half,\half)\in S^{\circ}(a),$ and with  $S^{\dagger}(a)\ne\emptyset.$
Set
\[A=\ff a+\ff b_0+\sum_{(\mu,\nu)\in S^{\circ}(a)}\ff b_{\mu,\nu}
\]
with $b_0\ne 0$, and
\[
\textstyle{b=a+b_0+\sum_{(\mu,\nu)\in S^{\circ}}b_{\mu,\nu},\quad
b_{\mu,\nu}^2=0, \forall (\mu,\nu)\in S^{\dagger}(a),\quad b_{\half,\half}^2= b_0,}
\]
\[
b_{\mu,\nu}b_{\mu',\nu'}=0,\text{ for }(\mu,\nu)\ne (\mu',\nu'),
\]
\[
\textstyle{b_0^2=0, \quad b_0b_{\mu,\nu}=b_{\mu,\nu}b_0=0,\ \forall (\mu,\nu)\in S^{\circ}(a).}
\] $A$ is a PAJ of type $\half.$
Note that  $\ff a+\ff b_0+\ff b_{\half,\half}=B(\half,1),$ in the notation of \cite{HRS}. Also, $b_0 A = A b_0 =0.$
%\item    Define the algebra $A= \lan \lan a,b \ran\ran$  satisfying $ab= \nu a + \mu b$ and $ba= \mu a + \nu b$, for
% $\mu+\nu = 1$, $\mu\ne \half$.   $A$ is of Jordan type~$\emptyset$.
  \smallskip

  \item \cite[Example~3.12(2)]{RoSe5}    $S^{\circ}(a)\subseteq\{(\mu,\nu)\in\ff\times\ff: \mu+\nu=1\},$   $ A=\ff a+ \sum_{(\mu,\nu)\in S} \ff b_{\mu,\nu},$   and $ S^{\dagger}(a) \ne \emptyset,$
         where  $b = a +\sum_{(\mu,\nu)\in S^{\circ}(a)} b_{\mu,\nu},$ and $b_{\mu,\nu}^2 = 0 $ for each $(\mu,\nu)\in S^{\circ}(a)$.
         $Z:=\sum \ff b_{\mu,\nu}$ is a square zero ideal in~$A.$

When $|S^{\circ}(a)|=1$ we have  \cite[Example~2.2]{RoSe5} in the case of a 2-generated algebra.
 \smallskip

         \item  The ``$S$-exceptional  algebra,'' \cite[Example~3.18]{RoSe5}
Suppose  $  S^{\circ}(a)\subseteq \{ (\mu,\nu) \in \ff \times \ff :   \mu+\nu =1\}$ with $ S^\dag(a) \ne\emptyset $, and let $A=\ff a + \ff b_0+ \sum_ {(\mu,\nu)\in S^{\circ}(a)} \ff b_{\mu,\nu},$ with $b_0\ne 0.$ We define  multiplication on~$A$ according to the following   rules,  with the sums taken over all $ (\mu,\nu)\in S(a)  $, where $b =  b_0+ \sum_ {(\mu,\nu)\in S(a)}  b_{\mu,\nu}$:

%\begin{equation}\label{eq111a}
%   ab = \sum \mu b_{\mu,\nu}, \qquad ba = \sum \nu b_{\mu,\nu},
%\end{equation}

\begin{equation}\label{eq112a}
   a b_{\mu,\nu} =  \mu b_{\mu,\nu} = b_{\mu,\nu} b, \qquad  b_{\mu,\nu}a =  \nu b_{\mu,\nu} = bb_{\mu,\nu}, \qquad  b_0^2 =b_0,
\end{equation}
\begin{equation}\label{eq113a}
    b_{\mu,\nu}  b_{\mu',\nu'}= 0, \quad \forall (\mu,\nu), (\mu',\nu') \in S(a).
\end{equation}

When $ (\half,\half)\in S(a)$, $A$ is of Jordan type $\half$; otherwise $A$ is    of Jordan type $\emptyset$.

%In fact if $(0,1), (1,0) \notin S$ then  $A$ is a PAJ in the sense of \cite{RoSe4}.
\smallskip

 \item  \cite[Example~4.12]{RoSe5}
 Let $T\subseteq\{(\mu,\nu)\in\ff\times\ff:\mu+\nu=1\},$ and let $S^{\circ}(a)=T\cup\{(2,2)\},$ with $S^{\dagger}(a)\ne\emptyset.$ Set
\[A=\ff a+\ff b_0+\ff b_{2,2}+\sum_{(\mu,\nu)\in T}\ff b_{\mu,\nu},
\]
with $b_0\ne 0,$ and
\[
\textstyle{b=a+b_0+\sum_{(\mu,\nu)\in S^{\circ}}b_{\mu,\nu},\quad
b_{\mu,\nu}^2=0, \forall (\mu,\nu)\in T,\quad b_{2,2}^2=-\half b_0.}
\]
\[
b_{\mu,\nu}b_{\mu',\nu'}=0,\text{ for }(\mu,\nu)\ne (\mu',\nu').
\]
\[
\textstyle{b_0^2=\frac{3}{2} b_0,\quad b_0b_{2,2}=b_{2,2}b_0=-\frac{3}{2}b_{2,2}, \quad b_0b_{\mu,\nu}=b_{\mu,\nu}b_0=0,\ \forall (\mu,\nu)\in T.}
\]
Clearly $bb_{\mu,\nu}=\mu b_{\mu,\nu},$ and $b_{\mu,\nu}b=\nu b_{\mu,\nu},\ \forall (\mu,\nu)\in T.$

$a$ and $b$ act symmetrically; both $a$ and $b$ are of Jordan type   2.
\smallskip

\item \cite[Example~4.9]{RoSe5}
% \end{example}
 $A=\ff a+\ff b_0+\sum_{(\mu,\nu)\in S^{\circ}}(a)\ff b_{\mu,\nu},$ with $\mu+\nu=1,$ for all  $(\mu,\nu)\in S^{\circ}(a),$ and $b_0\ne 0.$
Let
\[
b=a+b_0+\sum_{(\mu,\nu)\in S^{\circ}}b_{\mu,\nu}.
\]
\[
b_{\mu,\mu}b_{\mu',\nu'}=0,\quad\text{for all }(\mu,\nu), (\mu',\nu')\in S^{\circ}.
\]
\[
b_0^2=b_0,\quad b_0b_{\mu,\nu}=b_{\mu,\nu}b_0=0,\ \forall (\mu,\nu)\in S^{\circ}.
\]
$b$ is an  $\tilde S$-axis. $b$ is not weakly primitive  since $\dim A_1(b)=2,$ and $b$ is of Jordan type   1. We get the commutative
  \cite[Example~4.1]{RoSe5} (which is not an HRS algebra) when $S^\circ(a) = \{(\half,\half)\}.$
%  This is not a PAJ in the sense of  \cite{RoSe4}, since it is not one of the noncommutative PAJ's given there.
 \smallskip

\item \cite[Example~4.2]{RoSe5} Let $A=\ff a+\ff b_0+\ff b_{\mu,\mu},$ $b=\half a+b_0+b_{\mu,\mu},$
\[
\textstyle{b_{\mu,\mu}^2=\frac{1}{4\mu}a,\quad b_0^2=b_0+\frac{\mu-1}{4\mu}a,\quad b_0 b_{\mu,\mu} =\frac{1-\mu}{2}b_{\mu,\mu}.}
\]
$\dim A_0(b) =1$ and $\dim A_1(b)=2,$
so again  $b$ is not weakly primitive but is of Jordan type   1.    This is not  an HRS algebra, so is not a PAJ in the sense of  \cite{RoSe4}.
\smallskip

\item \cite[Example~4.3]{RoSe5}
 $A = \ff a+ \ff b_0+ \ff b_{\mu,\mu} +\ff b_{\nu,\nu},$ where $\mu,\nu\in\ff\setminus\{0,1\}, \mu\ne\nu,$ so $A$ is $4$-dimensional.  $A$ is commutative, and multiplication in $A$ is defined as follows.
\[
a^2=a,\quad ab_0=0,\quad ab_{\mu,\mu}=\mu b_{\mu,\mu},\quad ab_{\nu,\nu}=\nu b_{\nu,\nu}.
\]
To define $b_{\mu,\mu}^2$ and $b_{\nu,\nu}^2$ we use the following equation:
\[
\textstyle{b_{\mu,\mu}^2+b_{\nu,\nu}^2=\frac{1}{4}a+\half b_0,\quad\mu b_{\mu,\mu}^2+\nu b_{\nu,\nu}^2=\frac{1}{4} a.}
\]
We also let
\[
\textstyle{b_0^2=\half b_0,\quad b_0b_{\mu,\mu}=\frac{1-\mu}{2}b_{\mu,\mu},\quad b_0b_{\nu,\nu}=\frac{1-\nu}{2}b_{\nu,\nu},\quad b_{\mu,\mu}b_{\nu,\nu}=0.}
\]
$\dim A_0(b) = \dim A_1(b)=2,$  so $b$ is not weakly primitive but is  of Jordan type   $\{0,1\}$.
\end{enumerate}

\end{example}

These are the only possibilities, by the main theorems A and B of \cite{RoSe5}.

\begin{remark}\label{conc}
    In conclusion, in all of Examples \ref{class},
    \begin{enumerate}\eroman
         \item   $a$ is of Jordan type (i.e., \eqref{fr} holds) in all examples except (viii), in which case $b$ is not weakly primitive.

               \item  $\mu+\nu = 1$ and
 $A_{\mu,\nu}(a)^2 =0$ for all $(\mu,\nu)\in S^\dagger(a).$
           \item By symmetry, and in view of (i), when $b$ is weakly primitive, then $b$ is of Jordan type.
            \item $S^\dagger(b) = \tilde S^\dagger(a)$ or   $S^\dagger(b) =   S^\dagger(a)$.
         \item There are three commutative examples in which $\dim A_1(b)=2,$
         namely (vii), and (viii), and (ix).
    \end{enumerate}
\end{remark}

\begin{remark}\label{conca}
If $a$ is a special axis, and $b$ is an axis that commutes with $a$ then $A=\lan\lan a,b\ran\ran$ is commutative, by \cite[Lemma~3.7(ii)]{RoSe5}.
\end{remark}

\begin{cor}\label{2genJ}$ $
     The only  HRS algebras which extend to  weak  PAJ's which are not commutative are   $\ff,$ $\ff \times \ff,$ $B(\half,1)$ and its $2$-dimensional quotient, and $B(2,1).$
\end{cor}

 \begin{remark}
      The algebras of Examples \ref{class}  all are not isomorphic one to the other.  Indeed, the 2-dimensional algebras are all HRS or as in (iv) with $|S^\circ| = 1$.  We may eliminate the commutative 3-dimensional algebras, since the only commutative non-HRS algebras are as in (vi) or (viii).
      The noncommutative algebras (iv), (v), and (vi) are clearly different from each other since $S^\comm (a)$ differ, and these algebras are different from the rest of the list. One axis of  (vii)  is not weakly primitive, differentiating it from (iv) and (v).  (viii) is 3-dimensional commutative but not HRS, since one of its axes is not primitive.

      This leaves (ii) and (ix). We show that for $\mu=\half$ and $\nu=2,$ these examples are distinct.  Towards this end, suppose on the contrary that they are the same.  Then change the name of  $b$ of example (ix) to $c.$  Let $A$ be the algebra of example (ix).  $A$ contains the axes $a,b,c.$  Since $\dim A_{\gl}(a)=1,$ for $\gl\in\{1,0\}$ and $\dim A_{\eta,\eta}(a)=1,$ for $\eta\in\{\half,2\},$ we must have $\ff b_0=\ff c_0$ and $\ff b_{\half,\half}=\ff c_{\half,\half}.$ However $b_0b_{\half,\half}=0$ (by example (ii)) whereas $b_0c_{\half,\half}=\frac{1-\half}{2}c_{\half,\half}$ (by example (ix)); since $c_{\half,\half}\ne 0,$ this is a contradiction.
      %by \cite[Theorem~4.4]{RoSe4} if (ii) and (ix) were isomorphic, we could match up the axis $a$ in both examples, so $\mu = \half$ and $\nu = 2,$ and $b_{\half,\half}^2 = 0$ in (ix), implying
      %$0 =b_0 b_{\half,\half}  \in \{\frac{1-\half}2 , \frac{1-2}2\}
  %$ in (ii) a contradiction.
      %Towards this end by \cite[Theorem~4.4]{RoSe4} if (ii) and (ix) were isomorphic, we could match up the axis $a$ in both examples, and $b_0^2$ differs in the two algebras,
   %a contradiction.
    \end{remark}

%%%%%%%%%%%%%%%%%%%%%%%%%%%%%%%%%%%%%%%%%%%%%%
%%%%%%%%%%%%%%%%%%%%%%%%%%%%%%%%%%%%%%%%%%%%%
%sub1.2
\subsection{Miyamoto involutions}\label{Miyfhol01}$ $

%%%%%%%%%%%%%%%%%%%%%%%%%%%%%%%%%%%%%%%%%%%%%%%
%%%%%%%%%%%%%%%%%%%%%%%%%%%%%%%%%%%%%%%%%%%%%%%
\medskip
We modify the standard method of generating new axes. Let $a\in A$ be  any ${S}$-axis.
For $ {S'\subseteq S} , $ define the {\it Miyamoto map of $A$} $$\gt_{a;S'}: y_1+y_0+ \sum _{  (\mu,\nu)\in S\setminus S'} {y_{\mu,\nu}}\, -\sum _{(\mu,\nu)\in S'}y_{\mu,\nu}.$$
 Note  that  $\gt_{a;S'}(y)=y$  iff $y_{\mu,\nu}=0$ for all $(\mu,\nu)\in S'$.

\begin{lemma}\label{dec7} For any $S$-axis $a$ and any $y \in A$,
\begin{enumerate}\eroman
%i
\item $y_0, y_1, y_{\mu,\nu}\in \langle\langle a,y
\rangle\rangle$, for each $(\mu,\nu)$ in $S.$

 \item If $\gt_{a;\{(\mu,\nu)\}} (b)=b$ for all $(\mu,\nu)\in S,$  for a    weakly primitive axis $a$, and idempotent $b\ne a$, then $\dim  \langle\langle a,b\rangle\rangle =2.$ %If furthermore   $b$ is   a weakly primitive axis of Jordan type, then $ab=0=ba.$

\end{enumerate}
\end{lemma}
\begin{proof}  (i) By \cite[Lemma~1.7]{RoSe5}.

%(ii)  Follows at once from (i).

(ii)  $b = \ga_b a +b_0$. Thus $ab = \ga_b a =ba$, proving the  assertion.\qedhere

%Also $a$ is an $\ga_b$-eigenvector of $b.$  Hence $ \ga_b \in \{0,1\}$,   since $b$ is of Jordan type and $a^2=a$. But $\ga_b \ne 1$
%since $b$ is weakly primitive.
\end{proof}

 In accordance with the literature, a {\bf Miyamoto involution} is a Miyamoto map which is an algebra
automorphism.
\begin{lemma}
    Each  Miyamoto map associated to an $S$-axis $a$
satisfying the fusion rules   is a Miyamoto involution.
\end{lemma}
\begin{proof}
  $ \gt_{a;S'}$ is an   automorphism, because of the 2-grading in \eqref{dec119}.
\end{proof}

When $a$ does not satisfy the fusion rules,
$\gt_{a;S'}(b)$  need not be an idempotent, for an axis $b$!

 Write $\Aut(A)$ for the automorphisms of an algebra $A$.
For $\gr\in\Aut(A)$ we write $y^\gr$ for $\gr
(y)$, for any $y\in A$.

%%%%%%%%%%%%%%%%%%%%%%%%%%%%%%%%%%%%%%
%2.1
\begin{lemma}\label{abovename}
%%%%%%%%%%%%%%%%%%%%%%%%%%%%%%%%%%%%%%%%%
 Let $a$ be a weakly primitive  $S$-axis satisfying the fusion rules, and take $S'\subseteq S,$ and $\gr\in\Aut(A).$  Then \begin{enumerate} \eroman
    \item The axis $a^\rho $ has the same pre-Jordan  type $S$ as $a$;  $a^\rho $ is special if and only if $a$ is special.

 \item
$\gt_{a^{\gr};S'}=\gt_{a;S'}^{\gr}:=\gr^{-1}\gt_{a;S'}\gr.$
\end{enumerate}
\end{lemma}
\begin{proof} (i)
  Note  for
$(\mu,\nu)\in S,$ that
\[\tag{$*$}
x\in A_{\mu,\nu}(a^{\gr}) \iff x^{\gr^{-1}}\in A_{\mu,\nu}(a).
\]
\medskip

\noindent
(ii)
Let $\tau'=\tau_{a^{\gr};S'}$ and $\tau=\tau_{a;S'}.$ We show that $x^{\tau'}=\varepsilon x\implies x^{\tau^{\gr}}=\varepsilon x,$ for $\varepsilon\in\{1,-1\}.$  Clearly this implies $\tau'=\tau^{\gr}.$

Let $x\in A_1(a^{\gr})\cup A_0(a^{\gr})\cup A_{S\setminus S'}(a^{\gr}),$ (so that $x^{\tau'}=x$). Then by $(*),$ $x^{\gr^{-1}\tau}=x^{\gr^{-1}},$ so $x^{\tau^{\gr}}=x.$
 Also if $x\in A_{S'}(a^{\gr}),$ (i.e., $x^{\tau'}=-x$), then, by $(*)$, $x^{\gr^{-1}}\in A_{S'}(a),$ so $x^{\gr^{-1}\tau}=-x^{\gr^{-1}},$ so $x^{\tau^{\gr}}=x^{\gr^{-1}\tau\gr}=-x^{\gr^{-1}\gr}=-x.$
\end{proof}

\begin{example}
      We make the following observations on  Example \ref{class}(ii), our example of a 4-dimensional CPAJ,
    \begin{enumerate}\eroman
\item
$ \ff b_{\half,\half} $ is a square-zero ideal of $A$.

\item
$A = \langle\langle b, b^\gt_{a;S}
\rangle\rangle$ where $S =\{ \half, 2\}$, unless $\charc \ \ff=3.$ Indeed  let $B = \langle\langle b, b^\gt_{a;S'}
\rangle\rangle.$   Then $b_{2,2}+ b_{\half,\half}\in B,$
as is $b_0 +  a,$ so $-\half b_0 =(b_{2,2}+ b_{\half,\half})^2   \in B,$ implying $b_0 \in B,$ and $-\frac{3}{2}b_{2,2} = b_0(b_{2,2}+ b_{\half,\half})\in B,$ implying $B = A$.

\item
When $\charc\ \ff=3,$ $\ff b_0$ is an annihilating ideal of $A,$ and the image $(\ff \bar b_{2,2}+ \ff \bar b_{\half,\half} )/\ff b_0 $  is a square-zero ideal of $A/\ff b_0.$
\end{enumerate}
\end{example}

Given a set of axes $X$ satisfying the fusion rules, define $\mathcal{G}(X)$ to be the group  of automorphisms generated by  all  Miyamoto involutions  associated with the axes in $X$, and $\Miy(X)=\{ \gt(x): x\in X, \, \gt \in \mathcal{G}(X)\}$.

\begin{thm}[generalizing {\cite[Theorem 2.3]{RoSe4}}]\label{Kmt}
%%%%%%%%%%%%%%%%%%%%%%%%%%%%%%%%%%%%%%%%%%
Suppose $A$ is an axial algebra generated by a  set~$X$ of weakly
primitive axes satisfying the fusion rules.
\begin{enumerate}\eroman
%i
\item  $A$ is spanned by the set~$\Miy(X)$.

%ii
\item
If $W\subseteq A$ is a subspace containing $X$ such that $xv\in W$ and
 $vx\in W$ for all $v\in W$ and $x\in X,$ then $W=A.$

%iv
\item
Let $a\in A$ be an $S$-axis.  Then  for any $(\mu,\nu)\in S$ we have
\[
A_{\mu,\nu}(a)=\sum\{\ff b_{\mu,\nu} : b\in \Miy(X)\}.
\]
For example,
$$A_{0}(a)= \sum \{ \ff b_0 : b \in \Miy(X)\}.$$
\end{enumerate}

\item $A_{\mu,\nu}^2 = \sum\{\ff b_{\mu,\nu}  c_{\mu,\nu} : b,c\in \Miy(X)\} $
\end{thm}

\begin{proof} The proof is exactly as in the proof of \cite[Theorem 2.3]{RoSe4}. For example, we get (i) by noting   that $ab = \ga
_b a + \sum _{(\mu,\nu)\in S(a)} \frac{\mu}2 (b -  \gt_{a;\mu,\nu}(b))$ for any $a,b\in \Miy(X)$. Then repeat the remainder of the
proof of  {\cite[Theorem 2.3]{RoSe4}}, using $(\mu,\nu)$ instead of $(\gl,\gd).$
\end{proof}

\begin{cor}
Suppose $A=\lan\lan X\ran\ran,$ where each $x\in X$ is an axis  satisfying $S^{\dagger}(x)=\emptyset.$ Then $A$ is commutative.
\end{cor}
\begin{proof}
Each axis  $x\in \Cl(X)$ satisfies $S^{\dagger}(x)=\emptyset,$ so $xy=yx,$ for $x,y\in\Cl(X).$  Since $\Cl(X)$ spans $A,$  $A$ is commutative.
\end{proof}

\section{The existence of a Frobenius form}\label{Frob}
%%%%%%%%%%%%%%%%%%%%%%%%%%%%%%%%%%%%%%%%%%%%%%
%%%%%%%%%%%%%%%%%%%%%%%%%%%%%%%%%%%%%%%%%%%%
%%%%%%%%%%%%%%%%%%%%%%%%%%%%%%%%%%%%%%%%%%%

%%%%%%%%%%%%%%%%%%%%%%%%%%%%%%%%%%%%%%%%%%

\begin{Def} A {\it Frobenius form} on an algebra is a  symmetric  bilinear form which is associative, i.e., satisfying $(xy,z)= (x,yz)$ for all $x,y,z$.\footnote{Any associative bilinear form on a commutative axial algebra is symmetric, and thus a Frobenius form,  as seen in~\cite[Lemma~3.5]{HRS1}.}
\end{Def}

\begin{lemma}\label{Or}
Let $A$ be an algebra with a Frobenius form, and let $z\in A.$  Denote by $\mathfrak{R}$ the radical of the form. Then
\begin{enumerate}\eroman
\item $($\cite[Lemma 2.1]{S}$)$
If $u\in A_{\mu,\nu}(z)$ and $v\in A_{\mu',\nu'}(z),$ then $(u,v)=0,$ unless perhaps $\mu=\nu',$ and $\mu'=\nu.$

\item
If $A_{\mu,\nu}(z)A_{\nu,\mu}(z)=0,$ then $(A_{\mu,\nu}(z),A_{\nu,\mu}(z))=0,$ unless perhaps $\mu=\nu=0.$

\item
If $A_{\mu,\nu}(z)^2\subseteq \ff z+A_0(z),$ $(z,z)\ne 0,$ and $\mu\ne\nu,$ or $(\mu,\nu)=(0,0),$ then $A_{\mu,\nu}(z)^2\subseteq A_0(z).$

\item
If $z$ is an axis, and $A_{\mu,\nu}(z)A_{\nu,\mu}(z)=0,$ for each $(\mu,\nu)\in S^{\dagger}(z),$ then $\sum_{(\mu,\nu)\in S^{\dagger}(z)}A_{\mu,\nu}\subseteq \mathfrak{R}.$

%\item
%If $A_0(z)^2\subseteq \ff z+A_0(z),$ and $(z,z)\ne 0,$ then $A_0(z)^2\subseteq A_0(z).$

\item If $A$ is spanned by idempotents and $Az =0,$ then $z\in \mathfrak{R}$.

\end{enumerate}
\end{lemma}
\begin{proof}
(i)  $\mu(u,v)=(zu,v)=(v,zu)=(vz,u)=\nu'(u,v)$ so $(u,v)=0,$ unless $\mu=\nu'.$  By symmetry $(u,v)=0,$ unless $\mu'=\nu.$
\smallskip

\noindent
(ii) Let $u\in A_{\mu,\nu}(z)$ and $v\in A_{\nu,\mu}(z).$  Then $\mu(u,v)=(zu,v)=(z,uv)=0,$ so if $(u,v)\ne 0,$ $\mu=0.$  By symmetry if $(u,v)\ne 0,$ then $\nu=0.$
\smallskip

\noindent
(iii) Let $u,v\in A_{\mu,\nu}(z),$ then $(z,uv)=\mu(u,v)=0,$ by (1), so $uv\in A_0(z).$

\smallskip
\noindent
(iv)  This follows at once from (i) and (ii).

\smallskip
\noindent
(v) For any idempotent $b$ we have $(b,z) = (b^2,z)=(b,bz)=(b,0)=0.$
\end{proof}

% If an axial algebra $A$ has a Frobenius form, and $a\in A$ is an axis satisfying the fusion rules, then:
%    \begin{enumerate}\eroman
%        \item $(A_{\mu,\nu}(a), A_{\mu',\nu'}(a))=0,$ unless perhaps $\mu=\nu=\mu'=\nu';$
%        \item If $(a,a)\ne 0,$ then $A_0(a)^2\subseteq A_0(a).$
%    \end{enumerate}
%\end{lemma}
%\begin{proof}
%    (i) Let $y \in A_{\mu,\nu}(a)$ and  $z \in A_{\mu',\nu'}(a).$  If $(\mu,\nu)\ne (\mu',\nu'),$ then $\mu (y,z) =(ay, z) = (a,yz)=0,$ implying $(y,z)=0$ unless $\mu=0.$
%    Since $( , )$ is symmetric also $\mu'=0.$ Also $\nu(y,z)=(ya,z)=(y,az)=\mu'(a,z),$ implying $(y,z)=0$ unless $\nu=\mu'.$  But the same holds for $\mu'$ and $\nu'.$
%
%Suppose $(\mu,\nu)=(\mu',\nu'),$ then $\nu(y,z) = (ya,z) = (y,az) = \mu(y,z),$ so if $(y,z)\ne 0,$  then $\mu=\nu.$
%\end{proof}

Thus, \eqref{fr} is a necessary condition for a Frobenius form. But it superfluous in 2-generated PAJ's, as seen in Remark~\ref{conc},  and in type $\eta$ as seen in \cite[Theorem~3.2]{R}. This leads us to ask whether a pre-PAJ is necessarily a PAJ? We shall see this is often the case in Corollary~\ref{a0s}.

 The purpose of this section is to prove
the following theorem (extending \cite[Theorem 4.1, p.~407]{HSS}):

%%%%%%%%%%%%%%%%%%%%%%%%%%%%%%%%%%%%%%%%%%
%6.2
\begin{thm}\label{thm frob}
Suppose the weak PAJ $A$ is generated by a %minimal
homogeneous set~$X$ of special axes.

%$(-1,-1)\notin S^{\comm}$ and $(\eta,\eta)\in S^\comm$ implies $(1-\eta,1-\eta)\notin S^\comm,$
Then $A$ admits a Frobenius form, $(\ ,\ )$, which satisfies  $( a,y) =
\gvp_a(y)$, and hence $(a,a)=1,$  for all axes $a\in \Cl(X)$ and $y\in A.$ Furthermore $(xy,z)=(yx,z),$ for $x,y,z\in A.$
\end{thm}
 To prove the theorem, we shall see that it is enough to obtain a set  of axes $\widehat X$ spanning $A,$ for which $\varphi_{b}(c) = \varphi_{c}(b)$ for all $b,c\in\widehat  X.$
 Indeed, by Theorem~\ref{Kmt}(i), and Lemma \ref{sym} below, $\widehat X=\Cl(X)$ satisfies this property.
 %In Proposition~\ref{sym} below, we note that this is the case for all of our examples.

%%%%%%%%%%%%%%%%%%%%%%%%%%%%%%%%%%%%%
%5.1
\begin{prop}\label{sym}
%%%%%%%%%%%%%%%%%%%%%%%%%%%%%%%%%%
$\gvp_a(b)=\gvp_b(a)$ for any weakly primitive $S$-axes
$a,b\in A$ satisfying the fusion rules, under any of the following conditions:
 \begin{enumerate}\eroman
\item
$ab = ba = 0$, in which case $\gvp_a(b)=\gvp_b(a) = 0;$

\item
$a$ and $b$ do not commute;

\item
$\langle\langle a,b\rangle\rangle$ is commutative, of dimension 3, and $a,b$ are of the same type.
\item
$\langle\langle a,b\rangle\rangle$ is commutative, of dimension 4, and $a,b$ are of the same type.
\end{enumerate}
\end{prop}

\begin{proof} (i) The projections are both $0.$

(ii) The only relevant cases are Example~\ref{class}(iii),(iv),(vi) in which $\varphi_a(b) = \varphi_b(a) =1,$  and Example~\ref{class}(v) in which  $\varphi_a(b) = \varphi_b(a) =0. $

(iii) This appears in \cite{HRS}.

 (iv) By  Example~\ref{class}(ii),(vi), by symmetry of the presentation.
 \end{proof}

The other algebras of  Example~\ref{class} are not relevant, since the axis $b$ is not weakly primitive.

\begin{proof}[Proof of Theorem \ref{thm frob}]\hfill
\medskip

First note that by Theorem \ref{Kmt}(i) and by Lemma \ref{sym} for $b,c\in \Cl(X),$ $\varphi_b(c)=\varphi_c(b).$ As in the proof of \cite[Theorem~6.3]{RoSe4}, we start by defining a symmetric bilinear form $(\cdot\, ,\, \cdot)$ on a
 base $\calb\subseteq \Cl(X).$ We put $(a,b)=\varphi_a(b),$  for all $a,b\in\mathcal{B},$
which we extend by linearity to $A$
to get the symmetric bilinear form $(\cdot\, ,\,
\cdot)$.
We proceed by a sequence of steps.
%%%%%%%%%%%%%%%%%%%%%%%%%%%%%%%%%%%%%%
\begin{enumerate}
\item
$(a\, ,\, y)=\gvp_a(y),$ for a given axis $a\in \mathcal{B}$ and all $y\in A.$
Indeed, for $y =  \sum_{b\in\calb}\ga_b
b ,$ since $\gvp_a$ is linear,
\begin{gather*}\tag{$*$}
\gvp_a(y)=
\gvp_a\left(\sum_{b\in\calb}\ga_b b\right)
=\sum_{b\in\calb}\ga_b(a,b) =(a,\sum_{b\in\calb}\ga_b b)=(a,y).
\end{gather*}

Let now $a\in A$ be a weakly primitive axis satisfying the fusion rules and having the same type as the axes in $X.$ Then $\varphi_a(b)=\varphi_b(a)=(a,b),$ for all $b\in\mathcal{B}.$ Thus, as in $(*)$ we get,

\item
$(a\, ,\, y)=\gvp_a(y),$ for every weakly primitive  axis $a\in A$ satisfying the fusion rules and having the same type as the axes in $X,$ and all $y\in A;$
in particular
$(a\, ,\, a)=1.$

\item
$(\cdot\, ,\, \cdot)$ is invariant under any automorphism of $A$. In particular
$(\cdot\, ,\, \cdot)$ is invariant under any automorphism in $\calg(X).$
\end{enumerate}

   Let $\psi\in\Aut(A)$, and $a\in\Cl(X).$ If
\[
u=\gvp_a(u)a+u_0+\sum _{(\mu,\nu)\in S(a)} u_{\mu,\nu}
\]
is the decomposition of $u\in A$ with respect to $a,$
 then  the decomposition of $u^\psi$ with
respect to the primitive axis $a^\psi$
is $u^\psi=\gvp_a(u)a^\psi+u_0^\psi+\sum _{(\mu,\nu)\in S(a)} u_{\mu,\nu}^\psi$. Hence
$\gvp_{a^\psi}(u^\psi)=\gvp_a(u)$, and so, by (2), $(a^\psi,u^\psi)=(a,u)$.  Hence by linearity, also $(v,u)=(v^{\psi},u^{\psi}),$ for all $v,u\in A.$
Before concluding, we need a lemma  inspired by Lemma~\ref{Or}(i).

%%%%%%%%%%%%%%%%%%%%%%%%%%%%%%%%%%%%%%%%%%%%%%
%6.5
\begin{lemma}\label{O}
%%%%%%%%%%%%%%%%%%%%%%%%%%%%%%%%%%%%%%%%%%
For every  $a\in\Cl(X),$ the different eigenspaces of $a$ are
orthogonal with respect to $(\cdot\, ,\, \cdot)$.
\end{lemma}
\begin{proof}
Clearly, if $u\in A_0(a)+A_{S}(a),$ then $(a,u)=\gvp_a(u)=0$.
Hence $A_1(a)=\ff a$ is orthogonal to both $A_0(a)$ and to
$A_{S}(a)$. Assume $(0,0)\ne (\mu,\nu)\ne (\mu',\nu').$ If $u \in A_{\mu,\nu}(a)$ and  $v \in A_{\mu',\nu'}(a)$, take $\tau:=\tau_{a, S(a)\setminus\{(\mu,\nu)\}}.$  Then, by (3) above,  $(u,v)=(u^{\tau}, v^{\tau})=(-u,v),$ so $(u,v)=0.$\qedhere

%It remains to show that $A_0(a)$ orthogonal to
%$A_{\mu,\nu}(a)$. For $u\in A_0(a)$ and $v\in A_{\mu,\nu}(a)$, the
%fact that $(\cdot\, ,\, \cdot)$ is invariant under $\tau_a$ gives us
%$(u,v)=(u^{\tau_a},v^{\tau_a})=(u,-v)=-(u,v)$. Symmetry of the form implies $(u,v)=0$.
\end{proof}

Now, as in \cite{RoSe4}, it is easy to show associativity of the form by  checking  eigenspaces. Let $a\in\Cl(X).$ First we note that $(ay,z) = (y,az),$ for all $y,z\in A.$ Indeed, this is obvious if $y,z$ are in the same component with respect to $a$, and if $y\in A_{\mu,\nu}(a)$ and $z \in A_{\mu',\nu'}(a)$ for $(\mu,\nu)\ne (\mu',\nu')$, then both sides are in $( A_{\mu,\nu}(a), A_{\mu',\nu'}(a))=0.$ Thus by linearity $(ay,z)=(y,az),$ for all $y,z\in A.$ By linearity again, $(xy,z)=(y,xz),$ for all $x,y,z\in A.$

Similarly $(yx,z)=(y,zx).$  Now $(xy,z)=(y,xz)=(xz,y)=(x,yz).$ But also $(xy,z)=(x,zy)=(zx,y)=(yx,z).$
%
%By linearity, we may assume that $z$ is an axis. But now
%$(ay,z) = (y,az) = (yz,a) = (a,yz).$ By linearity, $(xy,z) = (x,yz)$ for all $x,y,z\in A.$
\end{proof}
%%%%%%%%%%%%%%%%%%%%%%%%%%%%%%%%%%%%%%

\begin{remark} $ $\begin{enumerate}\eroman
    \item   For any  weakly primitive axis $a$ satisfying the fusion rules, $(a,a)=1$ is equivalent to $(a,y)=\varphi_a(y),$ for all $y\in A,$ by Lemma \ref{Or}(i).

    This implies that the Frobenius form in Theorem \ref{thm frob} is uniquely defined under the condition that $(a,a)=1,$ for all $a\in\Cl(X).$ Indeed, if $y = \sum \gc_i a_i$ then $$(y,z) = \sum _i(\gc_i a_i,z) = \sum _i\gc_i ( a_i,z)= \sum \gc_i \gvp_{a_i}(z).  $$

\item     One might wonder whether the condition of homogeneity is needed. We know that all 2-generated weak PAJ's of Example~\ref{class} are homogeneous, but we do not know how to prove the theorem without a homogeneous generating set of axes, and with $3$-generated algebras we do not know how to modify the axes to get a homogeneous set.
\end{enumerate}\end{remark}
\begin{cor}
    \label{a0s} Under the hypotheses of Theorem \ref{thm frob}, a weakly primitive axis $a$ satisfying the fusion rules and $(a,a)\ne 0$ is of Jordan type.
\end{cor}
\begin{proof}
    By Lemma~\ref{Or}(iii).
\end{proof}
%%%%%%%%%%%%%%%%%%%%%%%%%%%%%%%%%%%%%%%%%

%%%%%%%%%%%%%%%%%%%%%%%%%%%%%%%%%%%%%%%%
%We can  normalize the Frobenius form unless some axis $b$ is isotropic, i.e., $(b,b)=0.$

\section{Applications to noncommutative PAJ's}

In view of Lemma~\ref{Or}(i),(ii),  $\sum_{(\mu,\nu) \in S^\dagger} A_{\mu,\nu}(a)$ is in the radical of the Frobenius form (if such a form exists).

\begin{remark}
    Let $A=\lan\lan X\ran\ran,$ where $X$ is a set of axes satisfying the fusion rules, such that $A$ possesses a Frobenius form  and $a, b\in X.$  Let
\[
Z_a:=\sum_{(\mu,\nu)\in S^\dag(a)}A_{\mu,\nu}(a).
\]
%Denote the radical of the Frobenius form as $\mathfrak{R}.$
\begin{enumerate}\eroman
\item
$Z_a\subseteq\mathfrak{R},$ by  Lemma~\ref{Or}(iv);
 \item If $A = \langle\langle a,b \rangle\rangle $ with $a,b$ special, then   $Z_a = Z_b$ and $Z_a^2 = 0,$   as seen by looking at the list in~Example~\ref{class}.
 Hence $Z_a$ is an ideal of $A,$ and
 $ A/Z_a$  is commutative.

\item  $Z_a^2a = 0,$   by Lemma \ref{Or}(iii). We noted in~Example~\ref{class}(iii) that one need not  have $ Az = 0$ implies $z=0.$
\end{enumerate}
\end{remark}

Let us address the following question:
\medskip

If $X$ is a set of axes satisfying the fusion rules,  is $Z_a^2 = 0$?
\medskip
\begin{lemma}
    To prove  $Z_a^2 = 0$ for axial algebras generated by an arbitrary number of special axes, it is enough to verify the assertion  for algebras generated by three special axes.
\end{lemma}
\begin{proof}
    In view of Theorem~\ref{Kmt}, it suffices to  prove $b_{\mu,\nu}c_{\mu,\nu}=0$ for axes $b,c.$
\end{proof}

%%%%%%%%%%%%%%%%%%%%%%%%%%%%%%%
%\begin{lemma}
%%%%%%%%%%%%%%%%%%%%%%%%%%
%Suppose that $A$ is generated by a set ~$X$ of strongly special axes, and assume that $Z_a^2\ne 0.$  Let $b\in X.$ Then
%\begin{enumerate}\eroman
%\item
%$b_{\mu,\mu}\in\mathfrak{R},$ for all $(\mu,\mu)\in S^{\circ}(a).$

%\item
%$\lan\lan a,b\ran\ran$ is as in Example \ref{class}(iii).
%\item
%$Z_a^2=\ff b_0.$

%\item
%$b_{\mu,\mu}\in \mathfrak{R},$ except perhaps if $\gvp_a(b)\ne 0,$ and $\mu=\frac{1}{2 \gvp_a(b)}.$
%\end{enumerate}
%\end{lemma}

\subsection{Noncommmutative examples}$ $

We conclude with some illustrative noncommutative examples which are not 2-generated,
generalizing the ``uniform'' and ``exceptional'' PAJ's of~\cite{RoSe5}.

\begin{example}[{\cite[Example~2.2]{RoSe5}}]\label{un1}
    Suppose   $\mu_{i,j}+ \nu_{i,j}=1,$ for $i,j\in I$.
    Define the algebra generated by a set of idempotents $E
= \{ x_i: i \in I\}$, satisfying $x_ix_j = \nu_{i,j} x_i + \mu_{i,j} x_j$ for
all $i\ne j \in I.$   All axes are special of Jordan type $\emptyset$.

When all the $\mu_{i,j}$ are equal, and  all the $\nu_{i,j}$ are equal, we get the ``uniform algebra'' of {\cite[Example~2.2]{RoSe5}}.
\end{example}

For the next two examples,   take sets $S^\dagger_{i,j} \ne\emptyset$   comprised of pairs $(\mu_{i,j}, \nu_{i,j}) : i \in I$ for which  $\mu_{i,j}+ \nu_{i,j}=1,$ and $S^\dagger = \cup_{i,j} S^\dagger_{i,j}.$ (To avoid complications in notation, one strictly should take $\mu_{i,j},\nu_{i,j} \ne \half,$ although the constructions also work in that case.)   Let $X=\{x_i:i\in I\}$ and $Y= \{y_i:i\in I\}$.
%$\langle\langle X \rangle\rangle$ and $\langle\langle Y \rangle\rangle$ are generalized uniform.

\begin{example}\label{exc1}
    (Generalizing \cite[Example~3.18]{RoSe5}).
    Define

     $A=\sum_{k\in I} \ff x_k +\sum \ff y_k+  \sum_{i,j}\sum_ {(\mu_i,\nu_j)\in S^\dag }\ff z_{\mu_i,\nu_j}$,
where we define  multiplication in~$A$ according to the following   rules,  with the sums taken over all the $ (\mu_i,\nu_j) \in S^\dag  $:

\begin{equation}\label{eq111}
   x_i y_j = \sum \mu_{i,j} z_{\mu_{i,j},\nu_{i,j}}, \qquad y_i x_j = \sum \nu_{i,j} z_{\mu_{i,j},\nu_{i,j}},
\end{equation}

\begin{equation}\label{eq112}
   x_i z_{\mu_{i,j},\nu_{i,j}} =  \mu_{i,j} z_{\mu_{i,j},\nu_{i,j}} = z_{\mu_{i,j},\nu_{i,j}} y_j, \qquad  z_{\mu_{i,j},\nu_{i,j}} x_i =  \nu_{i,j} z_{\mu_{i,j},\nu_{i,j}} = y_j z_{\mu_{i,j},\nu_{i,j}},
\end{equation}
\begin{equation}\label{eq113}
    z_{\mu_{i,j},\nu_{i,j}}   z_{\mu_{i',j'},\nu_{i',j'}} = 0, \quad \forall i,i',j,j'\in I.
\end{equation}

Each axis $x_i$ of $X$ is special of Jordan type $0$.
 By symmetry, each $y_i$ is special of Jordan type $0$.    We specialize to a more manageable example when all the $\mu_{i,j}$ are the same and all the $\nu_{i,j}$ are the same.    We specialize further to
 \cite[Example~3.21]{RoSe5} when $|S^\dagger|=1.$
\end{example}

\begin{example}\label{k2}
    (Generalizing
\cite[Example~4.12]{RoSe5}.
 Let  $S^{\circ}=S^{\dagger}\cup\{(2,2)\}.$
Let $A=\sum \ff x_i + \sum \ff y_{i,j;2}+ \sum_ {i,j\in I} \ff z_{\mu_{i,j},\nu_{i,j}}$, and
we put $$y_{i,j;0} = y_{j}
+y_{i,j;2}- \sum z_{\mu_{i,j},\nu_{i,j}},$$ and
we define  multiplication in~$A$ according to the following   rules,  taken over all $ (\mu_{i,j},\nu_{i,j})\in S ^\dag: $

\[
z_{\mu_{i,j},\nu_{i,j}}z_{\mu_{i',j'},\nu_{i',j'}}=0,\quad\forall i,j,i',j'\in I,
\]
\[
\textstyle{y_{i,j;0}^2=\frac{3}{2} y_{i,j;0},\quad y_{i,j;0}y_{i,j;2}=y_{i,j;2}y_{i,j;0}=-\frac{3}{2}y_{i,j;2},}\]\[\textstyle{ y_{i,j;0}b_{\mu,\nu}=b_{\mu,\nu}y_{i,j;0}=0,\ \forall (\mu,\nu)\in S^\dagger,}
\]
\[ x_i z_{\mu_i,\nu_j} = z_{\mu_i,\nu_j}y_j=\nu b_{\mu,\nu} \qquad  z_{\mu_i,\nu_j}x_i = y_j z_{\mu_i,\nu_j}=\mu_j z_{\mu_i,\nu_j},\ \forall (\mu_,\nu_j)\in S^\dag_{i,j}.\]

The axes $x_i$ and $y_j$ act symmetrically; both $x_i$ and $y_j$ are special of Jordan type   2.
\end{example}

\end{document}